\documentclass[reqno, 12pt]{amsart}

\newenvironment{nouppercase}{
  
  \renewcommand{\uppercasenonmath}[1]{}}{}

\usepackage{amsmath,amssymb,amsthm,amsfonts}
\usepackage{fullpage}
\usepackage{xcolor}
\usepackage[colorlinks=true,linkcolor=blue,anchorcolor=blue,citecolor=red]{hyperref}
\allowdisplaybreaks 
\usepackage{cite}

\newtheorem{theorem}{Theorem}

\newtheorem{lemma}{Lemma}

\newtheorem{corollary}{Corollary}

\newtheorem{counterexample}{Counterexample}

\newtheorem{proposition}[theorem]{Proposition}

\newtheorem*{definition}{Definition}

\newtheorem{algorithm}{Algorithm}

\author{Yuda Chen}
\address{School of Mathematical Sciences and LPMC, Nankai University, Tianjin 300071, China}
\email{Yuda Chen: tucydabc@foxmail.com}

\author{Xiangjun Dai}
\address{School of Mathematical Sciences and LPMC, Nankai University, Tianjin 300071, China}
\email{Xiangjun Dai:
2120230022@mail.nankai.edu.cn}

\author{Huixi Li}
\address{School of Mathematical Sciences and LPMC, Nankai University, Tianjin 300071, China}
\email{Huixi Li: lihuixi@nankai.edu.cn}

\date{\today}

\makeatletter
\@namedef{subjclassname@2020}{\textup{}2020 Mathematics Subject Classification}
\makeatother

\title[]{Some results on a conjecture of de Polignac about numbers of the form $p + 2^k$}

\begin{document}
	
\begin{abstract}
We have primarily obtained three results on numbers of the form $p + 2^k$. Firstly, we have constructed many arithmetic progressions, each of which does not contain numbers of the form $p + 2^k$, disproving a conjecture by Erd\H{o}s as Chen did recently. Secondly, we have verified a conjecture by Chen that any arithmetic progression that do not contain numbers of the from $p + 2^k$ must have a common difference which is at least 11184810. Thirdly, we have improved the existing upper bound estimate for the density of numbers that can be expressed in the form $p + 2^k$ to $0.490341088858244$.
\end{abstract}

\begin{nouppercase}
\maketitle
\end{nouppercase}

\section{Introduction}\label{Intro}

De Polignac \cite{dePolignac1849} conjectured in 1849 that every odd number larger than $3$ can be written as the sum of an odd prime and a power of $2$. Soon he realized that his conjecture is false, as $127$ and $959$ serve as easy counterexamples. It is interesting to study the density and distribution of numbers in the form $p + 2^k$ as well as those not in this form. Therefore, we consider
\[
d(N) = \frac{|\{n \leq N: n = p + 2^k, p \text{ prime}\}|}{N}.
\]
Let $\underline{d} = \liminf_{n \to \infty} d(n)$, and let $\overline{d} = \limsup_{n \to \infty} d(n)$.

With respect to $\underline{d}$, Romanoff \cite{Romanoff1934} proved in 1934 that $\underline{d} > 0$. A quantitative version of this result, $\underline{d} > 0.0868$, was given by Chen and Sun \cite{ChenSun2004} in 2004. Later, L\"{u} \cite{Lv2007} improved the constant to $0.09322$ in 2007, while Habsieger and Roblot \cite{HabsiegerRoblot2006} achieved a better result, $\underline{d} > 0.0933$, in 2006. Pintz \cite{Pintz2006} also contributed in 2006 with $\underline{d} > 0.093626$, followed by Habsieger and Sivak-Fischler \cite{HS2010} in 2010, who obtained $\underline{d} > 0.0936275$. The most recent improvement comes from Elsholtz and Schlage-Puchta \cite{ES2018} in 2018, with $\underline{d} > 0.107648$.

Concerning $\overline{d}$, van der Corput \cite{vanderCorput1950} demonstrated in 1950 that odd integers not of the form $p + 2^k$ possess a positive density. Also in 1950, Erd\H{o}s \cite{Erdos1950} provided a more explicit result: the arithmetic progression $7629217 \pmod{11184810}$ does not contain integers of the form $p + 2^k$. Erd\H{o}s' result implies that $\overline{d} \leq 0.5 - 1/11184810 < 0.49999991$. This estimate was further refined to $0.4909$ by Habsieger and Roblot \cite{HabsiegerRoblot2006} through algorithm design in 2006.

By employing direct computation and a probabilistic model by Bombieri \cite{Romani1978, Romani1983}, Romani \cite{Romani1983} made a conjecture in 1983 regarding the value of the Romanoff's constant $d = \overline{d} = \underline{d}$, if it exists, suggesting that $d \approx 0.434$. In 2020, Gianna del Corso, Ilaria del Corso, Dvornicich, and Romani \cite{DDDR2020} extended and enriched the methods in \cite{Romani1983} and conditionally proved some formulas claimed by Bombieri on which the probabilistic model is based. Their results suggest Romanoff's constant might be approximately $0.437$.

Our article presents three primary conclusions, all revolving around the upper density $\overline{d}$. The first finding centers on Erd\H{o}s's construction of the arithmetic progression $7629217 \pmod{11184810}$ that do not contain numbers of the form $p + 2^k$, where $11184810 = 2 \times 3 \times 5 \times 7 \times 13 \times 17 \times 241$. In his work \cite{Erdos1950}, Erd\H{o}s employed a concept known as the \emph{covering system}, which is a finite collection of arithmetic progressions whose union forms $\mathbb{Z}$. With the same approach, we have constructed many more arithmetic progressions that do not contain numbers of the form $p + 2^k$.

\begin{theorem}\label{apthm}
The 48 arithmetic progressions, $a \pmod{11184810}$, do not contain numbers of the form $p + 2^k$, where $a$ can be taken as the numbers in the last column of Tabel~\ref{aptable}.
Additionally, we have identified other arithmetic progressions with moduli different from $11184810$ that do not contain numbers of the form $p + 2^k$, such as
\begin{align*}
309547193 &\pmod{412729590}, \\
13982215829 &\pmod {21448163730} ,\\
520864019678683&\pmod {2520047004605130} ,\\
12878054009 &\pmod {44153328030} ,\\
154854279578189723614177 &\pmod{483570327845851669882470}.
\end{align*}
\end{theorem}

Theorem~\ref{apthm} is also related to a conjecture by Erd\H{o}s listed as Problem 16 of Bloom’s list \cite{Bloom}, which states that the set of positive odd integers that cannot be represented as the sum of a prime and a power of $2$ is the union of an infinite arithmetic progression of positive odd integers and a set of asymptotic density zero. Recently, Chen \cite{Chen2023} disproved this conjecture by constructing two arithmetic progressions $992077 \pmod{11184810}$ and $3292241 \pmod{11184810}$, both of which do not contain numbers of the form $p + 2^k$. Erd\H{o}s's arithmetic progression and Chen's arithmetic progressions are all included in our Theorem~\ref{apthm}.

In Chen's paper \cite{Chen2023}, he poses three questions related to the set $\mathcal{U}$ of positive odd integers that cannot be represented as $p + 2^k$. Let $A_i$, where $i \in I$, be a collection of all infinite arithmetic progressions of positive odd integers, none of which can be represented as the sum of a prime and a power of $2$, then Chen's first question is whether the set of positive odd integers not in the union of all $A_i$, where $i \in I$, has asymptotic density zero. Chen's second question is whether $b \geq 11184810$ if $a \pmod{b} \subset \mathcal{U}$. The third question Chen raises is whether $\mathcal{U}$ is a union of finitely many infinite arithmetic progressions of positive odd integers and a set of asymptotic density zero. Along with the arithmetic progressions in $\mathcal{U}$ listed in Theorem~\ref{apthm}, we have constructed infinitely many nontrivial covering systems in Proposition~\ref{infiap} at the end of Section~\ref{apsec}, which provide potential arithmetic progressions in $\mathcal{U}$. This partially solves Chen's Problem 3.

Furthermore, our second main result confirms a positive answer to Chen's Problem 2.

\begin{theorem}\label{Chenconjthm}
If an arithmetic progression $a \pmod{b}$ does not contain numbers of the form $p + 2^k$, then $b \geq 11184810$.
\end{theorem}

Besides, if an arithmetic progression $a \pmod{11184810}$ is in $\mathcal{U}$, then we have proved that it is one of the 48 arithmetic progressions listed in Theorem~\ref{apthm}.

Thirdly, by enhancing the algorithm proposed by Habsieger and Roblot, we have improved the existing upper bound estimate for $\overline{d}$ from $0.4909409303984105956480078184$ to $0.490341088858244$.

\begin{theorem}\label{upperboundthm}
We have $\overline{d} < 0.490341088858244$.
\end{theorem}

The paper is organized as follows.
In Section~\ref{apsec}, we explore content related to Theorem~\ref{apthm}. We discuss Erd\H{o}s' construction and its connection to covering systems, elaborate on the algorithm used in our study, and present our results alongside discussion.
In Section~\ref{Chenconjsec}, we describe how we have designed an algorithm to verify no arithmetic progressions $a \pmod{b}$ with $b < 11184810$ can be contained in $\mathcal{U}$, giving a proof of Theorem~\ref{Chenconjthm}. We also discuss the case $b = 11184810$ if $a \pmod{b}$ is in $\mathcal{U}$. Moving on to Section~\ref{upperboundsec}, our focus shifts to content associated with Theorem~\ref{upperboundthm}. We give an overview of Habsieger and Roblot's algorithm, elaborate on the enhancements made in our algorithm, and engage in discussion about our results. Our innovative contribution lies in the construction of many covering systems, as well as the design of algorithms and the utilization of computer assistance in our proofs.

\noindent \textbf{Notation. } Let $\mathcal{C} = \{a_1 \pmod{d_1}, a_2 \pmod{d_2}, \cdots, a_n \pmod{d_n} \}$ be a covering system with distinct moduli $d_1 < d_2 \cdots < d_n$. Let $D$ be the least common multiple of $d_1, d_2, \cdots, d_n$.
Let $\text{ord}_2(n)$ be the order of $2$ modulo an odd number $n$.
Let $d(n)$ be the number of (positive) divisors of $n$.
Let $p_{min}(n)$ be the smallest prime factor of $n$.

\section{Constructions of arithmetic progressions not containing numbers of the form $p + 2^k$}\label{apsec}

\subsection{Erd\H{o}s and Chen's constructions}

We first revisit Erd\H{o}s' construction \cite{Erdos1950} of the arithmetic progression $7629217 \pmod{11184810}$ in $\mathcal{U}$. Erd\H{o}s devised the following covering system for $\mathbb{Z}$:
\[
\{0 \pmod{2}, 0 \pmod{3}, 1 \pmod{4}, 3 \pmod{8}, 7 \pmod{12}, 23 \pmod{24} \}.
\]
By Fermat's little theorem we know that
$k \equiv 0 \pmod{2}$ and $x \equiv 1 \pmod{3}$ imply $3 \mid x - 2^k$,
$k \equiv 0 \pmod{3}$ and $x \equiv 1 \pmod{7}$ imply $7 \mid x - 2^k$,
$k \equiv 1 \pmod{4}$ and $x \equiv 2 \pmod{5}$ imply $5 \mid x - 2^k$,
$k \equiv 3 \pmod{8}$ and $x \equiv 8 \pmod{17}$ imply $17 \mid x - 2^k$,
$k \equiv 7 \pmod{12}$ and $x \equiv 11 \pmod{13}$ imply $13 \mid x - 2^k$,
and $k \equiv 23 \pmod{24}$ and $x \equiv 121 \pmod{241}$ imply $241 \mid x - 2^k$.
Therefore, by the Chinese Remainder Theorem we know for any natural number $k$ and any integer $x$ in the arithmetic progression $7629217 \pmod{11184810}$, where $11184810 = 2 \times 3 \times 5 \times 7 \times 13 \times 17 \times 241$, we have either $3$, or $7$, or $5$, or $17$, or $13$, or $241$ divides $x - 2^k$.
Moreover, such $x - 2^k$ do no equal to any of $3$, $5$, $7$, $13$, $17$ and $241$, since $x - 2^k \equiv 0, \text{ or } 4, \text{ or } 6 \pmod{7}$, but $3, 5, 17, 241 \not \equiv 0, \text{ or } 4, \text{ or } 6 \pmod{7}$,
$x - 2^k \equiv 0 \text{ or } 2 \pmod{3}$, while $7, 13 \not \equiv 0 \text{ or } 2 \pmod{3}$. Therefore, we know the arithmetic progression $7629217 \pmod{11184810}$ do not contain numbers of the form $p + 2^k$.

Similarly, Chen \cite{Chen2023} identified two additional covering systems for $\mathbb{Z}$:
\[
\{0 \pmod{2}, 1 \pmod{3}, 1 \pmod{4}, 3 \pmod{8}, 3 \pmod{12}, 23 \pmod{24} \},
\]
and
\[
\{1 \pmod{2}, 0 \pmod{3}, 0 \pmod{4}, 2 \pmod{8}, 2 \pmod{12}, 22 \pmod{24} \},
\]
resulting in two corresponding arithmetic progressions $992077 \pmod{11184810}$ and $3292241 \pmod{11184810}$. Numbers $x$ in these arithmetic progressions satisfy that either $3$, $5$, $7$, $13$, $17$, or $241$ divides $x - 2^k$ for all natural numbers $k$. Moreover, by considering the remainders of $x - 2^k$ modulo $3$, $7$, and $17$, we can deduce that such $x - 2^k$ are not equal to any of the aforementioned 6 primes.
In terms of Erd\H{o}s' conjecture mentioned in Section~\ref{Intro}, since $\gcd(11184810, 992077 - 3292241) = 2$, we can infer that $\mathcal{U}$ is not the union of an infinite arithmetic progression of positive odd integers and a set of asymptotic density zero. Otherwise, $\mathcal{U}$ would contain all sufficiently large odd integers.

\subsection{Theoretical framework and algorithm design} Our first algorithm relies on the following lemma.
\begin{lemma}\label{apstrategy}
Let $\mathcal{C} = \{a_1 \pmod{d_1}, a_2 \pmod{d_2}, \cdots, a_n \pmod{d_n} \}$ be a covering system with distinct moduli $d_1 < d_2 \cdots < d_n$. If for $1 \leq i \leq n$, the numbers $2^{d_i} - 1$ have distinct prime factors $p_i$, then there exists an arithmetic progression $a \pmod{M}$, where $M = 2 \prod_{i = 1}^n p_i$, such that for all natural numbers $k$ and all numbers $x$ in $a \pmod{M}$, we have $x - 2^k$ is divisible by some of the $p_i$, where $1 \leq i \leq n$.
\end{lemma}
\begin{proof}
Since $\mathcal{C}$ is a covering system, any natural number $k$ satisfies at least one of the congruences $k \equiv a_i \pmod{d_i}$, where $1 \leq i \leq n$. By Fermat's little theorem we know if $x \equiv 2^{a_i} \pmod{p_i}$, then $p_i \mid x - 2^k$. Therefore, our claim follows by the Chinese remainder theorem.
\end{proof}

While the arithmetic progressions in $\mathcal{U}$ may not necessarily stem from the arithmetic progressions outlined in Lemma~\ref{apstrategy}, Lemma~\ref{apstrategy} still provides us with potential choices of arithmetic progressions in $\mathcal{U}$. Therefore, we will assign some names to such covering systems and arithmetic progressions.

\begin{definition}
Let $\mathcal{C} = \{a_1 \pmod{d_1}, a_2 \pmod{d_2}, \cdots, a_n \pmod{d_n} \}$ be a covering system with distinct moduli $d_1 < d_2 \cdots < d_n$. If for some distinct primes $p_1, p_2, \cdots, p_n$ such that $p_i \mid 2^{d_i} - 1$ for $1 \leq i \leq n$, then we say $\mathcal{C}$ is a \emph{CDL covering system}. The numbers satisfying the congruences $x \equiv 1 \pmod{2}$ and $x \equiv 2^{a_i} \pmod{p_i}$, where $1 \leq i \leq n$, form an arithmetic progression called a \emph{CDL arithmetic progression}.
\end{definition}

The following theorem of Bang \cite{Bang1886} gives a criteria of which covering system with distinct moduli is also a \emph{CDL covering system} with distinct moduli.
\begin{theorem}[Bang]\label{Bangthm}
For any integer $m > 1$ and $m \neq 6$, there exists a prime $p$ such that $p$ divides $2^m - 1$ and $p$ does not divide $2^{\tilde{m}} - 1$ for any $\tilde{m} < m$.
\end{theorem}
\begin{corollary}\label{Bangcoro}
Every covering system with distinct moduli and without $2$, $3$, and $6$ as simultaneous moduli in it is a \emph{CDL covering system}.
\end{corollary}
\begin{proof}
We simply need to notice that $2^2 - 1 = 3$ is a prime, $2^3 - 1 = 7$ is a prime, and $2^6 - 1 = 63 = 3^2 \times 7$, then our corollary follows by Theorem~\ref{Bangthm}.
\end{proof}

Recall that a covering system is called \emph{minimal} if none of the congruence classes is redundant. Similarly we can define \emph{minimal CDL arithmetic progressions}. For the purpose of finding nontrivial arithmetic progressions in $\mathcal{U}$, it is evident that we should focus on \emph{minimal CDL covering systems}.

Next, we explain the process of finding \emph{minimal CDL covering systems}. We consider $D = \text{lcm}(d_1, d_2, \cdots, d_n)$ in ascending order. By considering the density of integers covered by the congruence classes, we observe that $\sum_{i = 1}^n \frac{1}{d_i} > 2$, implying that a \emph{CDL covering system} must satisfy $D \geq 12$. When $D = 12$, we must have $n = 5$, $d_1 = 2$, $d_2 = 3$, $d_3 = 4$, $d_4 = 6$, and $d_5 = 12$. However, we can not pick pairwise distinct prime divisors of $2^2 - 1$ and $2^3 - 1$, $2^4 - 1$, $2^6 - 1$, and $2^{12} - 1$. Then we consider the cases where $D = 18$, $D = 20$, $D = 24$, and so forth.  Now we state our algorithm based on Lemma~\ref{apstrategy} to find all \emph{minimal CDL covering systems} where the least common multiple of the moduli of the congruence classes equals $D$.

\begin{algorithm}\label{apalgorithm}
We select a number $D$ such that $\sum_{d \mid D} \frac{1}{d} > 2$ and proceed with the following 4 steps.
\begin{enumerate}
    \item List all positive divisors of $D$ in ascending order as $1 = e_{1} < e_{2} < \cdots < e_{d(D)}$.
    \item List all tuples $(d_{1}, \cdots, d_{n})$ of divisors of $D$ such that $\sum_{i = 1}^n \frac{1}{d_i} > 1$ and $2^{d_i} - 1$ has distinct prime divisors, where $1 \leq n \leq d(D)$.
    \item For the n-tuples $(m_1, \cdots, m_n)$, where $0 \leq m_i \leq d_i - 1$ and $1 \leq i \leq n$, from $\{0, 1, 2, \cdots, D - 1\}$, we remove the residues congruent to $m_i$ modulo $d_i$ for $1 \leq i \leq n$. If all residues are removed, then we have found a \emph{CDL covering system}.
    \item Check whether the \emph{CDL covering system} found in Step $(3)$ is minimal or not. If it is not minimal, exclude it from consideration. Compute the corresponding \emph{CDL arithmetic progressions} for the \emph{minimal CDL covering systems}.
\end{enumerate}
\end{algorithm}

\subsection{Our results and some discussion on Theorem~\ref{apthm}}

With Algorithm~\ref{apalgorithm}, we did not find any \emph{CDL covering systems} with $D \leq 23$. However, we discovered 96 \emph{minimal CDL covering systems} with $D = 24$, corresponding to 48 \emph{CDL arithmetic progressions}. Refer to Table~\ref{aptable} and Table~\ref{aptablemod6}. Throughout the rest of the paper, we refer to \emph{minimal CDL covering systems} simply as \emph{CDL covering systems}, as these are the ones we are primarily interested in.

\scriptsize

\begin{table}
    \centering
    \begin{tabular}{|c|c|c|c|c|c|c|}
        \hline
        $\pmod{2}$ & $\pmod{3}$ & $\pmod{4}$ & $\pmod{8}$ & $\pmod{12}$ & $\pmod{24}$ & $a \pmod{11184810}$ \\ \hline
        0 & 0 & 1 & 3 & 7 & 23 & 7629217 \\ \hline
        0 & 1 & 1 & 3 & 3 & 23 & 992077 \\ \hline
        1 & 2 & 0 & 2 & 6 & 22 & 6610811 \\ \hline
        1 & 0 & 0 & 2 & 2 & 22 & 3292241 \\ \hline
        0 & 1 & 3 & 1 & 5 & 21 & 509203 \\ \hline
        0 & 2 & 3 & 1 & 1 & 21 & 4442323 \\ \hline
        1 & 0 & 2 & 0 & 4 & 20 & 8643209 \\ \hline
        1 & 1 & 2 & 0 & 0 & 20 & 10609769 \\ \hline
        0 & 0 & 1 & 7 & 11 & 19 & 8101087 \\ \hline
        0 & 2 & 1 & 7 & 3 & 19 & 7117807 \\ \hline
        1 & 2 & 0 & 6 & 10 & 18 & 1254341 \\ \hline
        1 & 1 & 0 & 6 & 2 & 18 & 762701 \\ \hline
        0 & 1 & 3 & 5 & 9 & 17 & 3423373 \\ \hline
        0 & 0 & 3 & 5 & 1 & 17 & 3177553 \\ \hline
        1 & 0 & 2 & 4 & 8 & 16 & 4507889 \\ \hline
        1 & 2 & 2 & 4 & 0 & 16 & 4384979 \\ \hline
        0 & 1 & 1 & 3 & 11 & 15 & 10581097 \\ \hline
        0 & 2 & 1 & 3 & 7 & 15 & 5050147 \\ \hline
        1 & 0 & 0 & 2 & 10 & 14 & 8086751 \\ \hline
        1 & 1 & 0 & 2 & 6 & 14 & 10913681 \\ \hline
        0 & 2 & 3 & 1 & 9 & 13 & 1247173 \\ \hline
        0 & 0 & 3 & 1 & 5 & 13 & 8253043 \\ \hline
        1 & 1 & 2 & 0 & 8 & 12 & 3419789 \\ \hline
        1 & 2 & 2 & 0 & 4 & 12 & 1330319 \\ \hline
        0 & 0 & 1 & 7 & 7 & 11 & 4506097 \\ \hline
        0 & 1 & 1 & 7 & 3 & 11 & 9053767 \\ \hline
        1 & 2 & 0 & 6 & 6 & 10 & 5049251 \\ \hline
        1 & 0 & 0 & 6 & 2 & 10 & 1730681 \\ \hline
        0 & 1 & 3 & 5 & 5 & 9 & 10913233 \\ \hline
        0 & 2 & 3 & 5 & 1 & 9 & 3661543 \\ \hline
        1 & 0 & 2 & 4 & 4 & 8 & 8252819 \\ \hline
        1 & 1 & 2 & 4 & 0 & 8 & 10219379 \\ \hline
        0 & 0 & 1 & 3 & 11 & 7 & 2313487 \\ \hline
        0 & 2 & 1 & 3 & 3 & 7 & 1330207 \\ \hline
        1 & 2 & 0 & 2 & 10 & 6 & 9545351 \\ \hline
        1 & 1 & 0 & 2 & 2 & 6 & 9053711 \\ \hline
        0 & 1 & 3 & 1 & 9 & 5 & 1976473 \\ \hline
        0 & 0 & 3 & 1 & 1 & 5 & 1730653 \\ \hline
        1 & 0 & 2 & 0 & 8 & 4 & 3784439 \\ \hline
        1 & 2 & 2 & 0 & 0 & 4 & 3661529 \\ \hline
        0 & 1 & 1 & 7 & 11 & 3 & 4626967 \\ \hline
        0 & 2 & 1 & 7 & 7 & 3 & 10280827 \\ \hline
        1 & 0 & 0 & 6 & 10 & 2 & 10702091 \\ \hline
        1 & 1 & 0 & 6 & 6 & 2 & 2344211 \\ \hline
        0 & 2 & 3 & 5 & 9 & 1 & 2554843 \\ \hline
        0 & 0 & 3 & 5 & 5 & 1 & 9560713 \\ \hline
        1 & 1 & 2 & 4 & 8 & 0 & 9666029 \\ \hline
        1 & 2 & 2 & 4 & 4 & 0 & 7576559 \\ \hline
    \end{tabular}
\caption{\scriptsize{\emph{CDL Covering systems} with $D = 24$ and their corresponding \emph{CDL arithmetic progressions}}}
\label{aptable}
\end{table}

\begin{table}
    \centering
    \begin{tabular}{|c|c|c|c|c|c|c|}
        \hline
        $\pmod{2}$ & $\pmod{6}$ & $\pmod{4}$ & $\pmod{8}$ & $\pmod{12}$ & $\pmod{24}$ & $a \pmod{11184810}$ \\ \hline
        0 & 3 & 1 & 3 & 7 & 23 & 7629217 \\ \hline
        0 & 1 & 1 & 3 & 3 & 23 & 992077 \\ \hline
        1 & 2 & 0 & 2 & 6 & 22 & 6610811 \\ \hline
        1 & 0 & 0 & 2 & 2 & 22 & 3292241 \\ \hline
        0 & 1 & 3 & 1 & 5 & 21 & 509203 \\ \hline
        0 & 5 & 3 & 1 & 1 & 21 & 4442323 \\ \hline
        1 & 0 & 2 & 0 & 4 & 20 & 8643209 \\ \hline
        1 & 4 & 2 & 0 & 0 & 20 & 10609769 \\ \hline
        0 & 3 & 1 & 7 & 11 & 19 & 8101087 \\ \hline
        0 & 5 & 1 & 7 & 3 & 19 & 7117807 \\ \hline
        1 & 2 & 0 & 6 & 10 & 18 & 1254341 \\ \hline
        1 & 4 & 0 & 6 & 2 & 18 & 762701 \\ \hline
        0 & 1 & 3 & 5 & 9 & 17 & 3423373 \\ \hline
        0 & 3 & 3 & 5 & 1 & 17 & 3177553 \\ \hline
        1 & 0 & 2 & 4 & 8 & 16 & 4507889 \\ \hline
        1 & 2 & 2 & 4 & 0 & 16 & 4384979 \\ \hline
        0 & 1 & 1 & 3 & 11 & 15 & 10581097 \\ \hline
        0 & 5 & 1 & 3 & 7 & 15 & 5050147 \\ \hline
        1 & 0 & 0 & 2 & 10 & 14 & 8086751 \\ \hline
        1 & 4 & 0 & 2 & 6 & 14 & 10913681 \\ \hline
        0 & 5 & 3 & 1 & 9 & 13 & 1247173 \\ \hline
        0 & 3 & 3 & 1 & 5 & 13 & 8253043 \\ \hline
        1 & 4 & 2 & 0 & 8 & 12 & 3419789 \\ \hline
        1 & 2 & 2 & 0 & 4 & 12 & 1330319 \\ \hline
        0 & 3 & 1 & 7 & 7 & 11 & 4506097 \\ \hline
        0 & 1 & 1 & 7 & 3 & 11 & 9053767 \\ \hline
        1 & 2 & 0 & 6 & 6 & 10 & 5049251 \\ \hline
        1 & 0 & 0 & 6 & 2 & 10 & 1730681 \\ \hline
        0 & 1 & 3 & 5 & 5 & 9 & 10913233 \\ \hline
        0 & 5 & 3 & 5 & 1 & 9 & 3661543 \\ \hline
        1 & 0 & 2 & 4 & 4 & 8 & 8252819 \\ \hline
        1 & 4 & 2 & 4 & 0 & 8 & 10219379 \\ \hline
        0 & 3 & 1 & 3 & 11 & 7 & 2313487 \\ \hline
        0 & 5 & 1 & 3 & 3 & 7 & 1330207 \\ \hline
        1 & 2 & 0 & 2 & 10 & 6 & 9545351 \\ \hline
        1 & 4 & 0 & 2 & 2 & 6 & 9053711 \\ \hline
        0 & 1 & 3 & 1 & 9 & 5 & 1976473 \\ \hline
        0 & 3 & 3 & 1 & 1 & 5 & 1730653 \\ \hline
        1 & 0 & 2 & 0 & 8 & 4 & 3784439 \\ \hline
        1 & 2 & 2 & 0 & 0 & 4 & 3661529 \\ \hline
        0 & 1 & 1 & 7 & 11 & 3 & 4626967 \\ \hline
        0 & 5 & 1 & 7 & 7 & 3 & 10280827 \\ \hline
        1 & 0 & 0 & 6 & 10 & 2 & 10702091 \\ \hline
        1 & 4 & 0 & 6 & 6 & 2 & 2344211 \\ \hline
        0 & 5 & 3 & 5 & 9 & 1 & 2554843 \\ \hline
        0 & 3 & 3 & 5 & 5 & 1 & 9560713 \\ \hline
        1 & 4 & 2 & 4 & 8 & 0 & 9666029 \\ \hline
        1 & 2 & 2 & 4 & 4 & 0 & 7576559 \\ \hline
    \end{tabular}
\caption{\footnotesize{\emph{CDL Covering systems} with $D = 24$ and their corresponding \emph{CDL arithmetic progressions}}}
\label{aptablemod6}
\end{table}

\normalsize

Next we prove these 48 \emph{CDL arithmetic progressions} modulo $11184810$ do not contain numbers of the form $p + 2^k$. It suffices to show such arithmetic progressions do no contain primes in $\{3, 5, 7, 13, 17, 241\}$. As Erd\H{o}s and Chen did, we consider the equations $11184810 x + a - 2^k = c$ modulo $3$, $5$, $7$, $13$, $17$, and $241$, respectively, where $a$ is in the last column of Table~\ref{aptable} and $c$ is in the set $\{3, 5, 7, 13, 17, 241\}$. This allows us to prove 41 of the 48 arithmetic progressions do not contain the primes $3$, $5$, $7$, $13$, $17$, and $241$. We still need to argue the other 7 arithmetic progressions do not contain these primes.

Let us take the equation $11184819x + 7576669 - 2^k = c$ , where $c \in {3, 5, 7, 13, 17, 241}$, as an example. This equation is solvable modulo $3$, $5$, $7$, $13$, $17$, and $241$,  but it is actually unsolvable modulo $85$. Analyzing the equations in this way can be cumbersome, so we take a more direct approach. In fact, for the remaining 7 arithmetic progressions, we can directly consider the corresponding equations modulo $11184810$ and conclude that the primes $3$, $5$, $7$, $13$, $17$, and $241$ do not appear in these arithmetic progressions. The computation is not heavy, because for each equation, there are only 24 values of $k$ to check, given that $11184810/2$ divides $2^{24} - 1$.

After this simple verification we can conclude the 48 \emph{CDL arithmetic progressions} modulo $11184810$ corresponding to \emph{CDL covering systems} with $D = 24$ in Table~\ref{aptable} and Table~\ref{aptablemod6} are all in $\mathcal{U}$. Is it true that all of the arithmetic progressions modulo $11184810$ in $\mathcal{U}$ are already listed? The answer is yes, and we will explain this in the last paragraph of Section~\ref{Chenconjsec}.

Another remark is that, as a sufficient condition in Chen's proof of Erd\H{o}s' conjecture, among the $48 \times 47 /2 = 1128$ pairs of arithmetic progressions $a_i \pmod{11184810}, a_j \pmod{11184810}$, we can check that 768 of them satisfy $(11184810, a_i - a_j) = 2$.

For larger values of $D$,  if we only aim to find one \emph{CDL covering system} with the given $D$ instead of enumerating all possibilities, we can slightly modify our Algorithm~\ref{apalgorithm} to ease the computation. Below, we provide additional examples of \emph{CDL covering systems}
$\{m_1 \pmod{d_1}, \cdots, m_n \pmod{d_n} \}$ and their corresponding \emph{CDL arithmetic progressions} in Table~\ref{morecoveringsystems}. Similarly, we can verify that these arithmetic progressions also exclude their corresponding prime values. Hence, they are also contained in $\mathcal{U}$.

\tiny
\begin{table}
    \centering

    \begin{tabular}{|c|c|c|c|}
    \hline $D$&$\{d_{1}, d_{2}, \cdots, d_{n}\}$&$\{m_{1}, m_{2}, \cdots, m_{n}\}$&{\tiny corresponding \emph{CDL arithmetic progressions}}\\
    \hline 36&\{2,3,4,9,12,18,36\}&\{1,2,3,8,11,17,35\}&{\tiny $309547193 \pmod{412729590}$}\\
    \hline48&\{2,4,6,8,16,24,48\}&\{1,2,0,0,4,4,44\}&{\tiny $13982215829 \pmod {21448163730}$}\\
    \hline60&\{2,3,4,5,10,12,15,20,30,60\}&\{0,1,3,3,5,9,11,17,29,59\}&{\tiny $520864019678683\pmod {2520047004605130}$}\\
    \hline72&\{2,4,6,8,18,24,36,72\}&\{1,2,4,6,16,22,34,70\}&{\tiny $12878054009 \pmod {44153328030}$}\\
    \hline80&\{2,4,5,8,10,16,20,40,80\}&\{0,1,3,3,7,7,15,31,79\}&{\tiny $154854279578189723614177 \pmod{483570327845851669882470}$}\\
    \hline
    \end{tabular}  \\
    \caption{More \emph{CDL covering systems} and \emph{CDL arithmetic progressions}}
    \label{morecoveringsystems}
\end{table}

\normalsize

Moreover, we can construct infinitely many \emph{CDL covering systems} corresponding to infinitely many nontrivial \emph{CDL arithmetic progressions}. If we could further demonstrate that these \emph{CDL arithmetic progressions} do not take their corresponding prime values, this would imply a negative answer to Chen's Problem 3.

\begin{proposition}\label{infiap}
Let $\{m_1 \pmod{d_1}, m_2 \pmod{d_2}, \cdots, m_n \pmod{d_n} \}$ be a minimal covering system with distinct moduli such that $\text{lcm}(d_1, d_2, \cdots, d_n) = D$. Then $\{1 \pmod{2}, 2m_1 \pmod{2d_1}, 2m_2 \pmod{2d_2}, \cdots, 2m_n \pmod{2d_n}\}$ is a minimal \emph{CDL covering system} with distinct moduli such that $\text{lcm}(2, 2d_1, 2d_2, \cdots, 2d_n) = 2D$.
\end{proposition}
\begin{proof}
Given a minimal covering system $\{m_1 \pmod{d_1}, m_2 \pmod{d_2}, \cdots, m_n \pmod{d_n} \}$ with distinct moduli such that $\text{lcm}(d_1, d_2, \cdots, d_n) = D$, it is easy to check $\{1 \pmod{2}, 2m_1 \pmod{2d_1}, 2m_2 \pmod{2d_2}, \cdots, 2m_n \pmod{2d_n}\}$ is a minimal covering system with distinct moduli such that $\text{lcm}(2, 2d_1, 2d_2, \cdots, 2d_n) = 2D$. It remains to prove we can pick distinct prime divisors of $2^2 - 1$, $2^{2d_1} - 1$, $2^{2d_2} - 1$, $\cdots$, and $2^{2d_n} - 1$. As $3$ is not in $\{2, 2d_1, 2d_2, \cdots, 2d_n\}$, our proposition follows by Corollary~\ref{Bangcoro}.
\end{proof}

\section{Verification of a conjecture by Chen}\label{Chenconjsec}
Recall that Chen's Problem 2 asks whether it is true that $b \geq 11184810$ when $a \pmod{b} \subset \mathcal{U}$? In this section, we provide a positive answer to this question. It suffices to prove that for any positive even integer $b$ and odd integer $a$ with $0 < a < b < 11184810$, there exists some $x = p + 2^k$ such that $x \equiv a \pmod{b}$.

A natural approach is to identify a natural number $x$ of the form $p + 2^k$ in every arithmetic progression $a \pmod{b}$ we aim to exclude. However, conducting such exhaustive exclusions would require thousands of hours of computation. Therefore, we have redesigned an algorithm based on Dirichlet's theorem on primes in arithmetic progressions.

\begin{lemma}\label{ChenQ2lem}
Let $b$ be an even natural number. If for some natural number $m$, the set of odd residues modulo $b$ is a union of $m$ shifts of the reduced residue systems by $2^k$ modulo $b$, where $1 \leq k \leq m$, then $a \pmod {b}$ is not an arithmetic progression in $\mathcal{U}$.
\end{lemma}
\begin{proof}
Suppose $a \pmod{b}$ is in $\mathcal{U}$, then $a$ is odd and $b$ is even. Suppose for some natural number $m$ we have
\[
\{1, 3, \cdots, b - 1\} = \bigcup_{k = 1}^m \{b_1 + 2^i, \cdots, b_{\varphi(b)} + 2^k\},
\]
where $\{b_1, \cdots, b_{\varphi(b)}\}$ is the reduced residue system modulo $b$. Then for every odd residue $1 \leq j \leq b - 1$ modulo $b$, we have $j - 2^k \in \{b_1, \cdots, b_{\varphi(b)}\}$ for some $1 \leq k \leq m$, i.e., we have $(j - 2^k, b) = 1$ for some $1 \leq k \leq m$. By Dirichlet's theorem on primes in arithmetic progressions, we know $j - 2^k + b y = p$ for some positive integer $y$ and some prime $p$, i.e., the arithmetic progression $j \pmod{b}$ contains a number of the form $p + 2^k$. This is a contradiction to $a \pmod{b}$ is in $\mathcal{U}$.
\end{proof}

Let $b = 2^j b'$, where $b'$ is the largest odd divisor of $b$. Since $2^{\text{ord}_2(b')} \equiv 1 \pmod{b'}$, it is adequate to verify whether the condition stated in Lemma~\ref{ChenQ2lem} holds for some $m \leq j + \text{ord}_2(b')$. Now we explain our algorithm based on Lemma~\ref{ChenQ2lem}.

\begin{algorithm}\label{ChenQ2algorithm}
We start with the even integer $b = 2$ and apply the following process to all even integers less than $11184810$.
\begin{enumerate}
    \item For an even integer $b$, compute the set of reduced residue systems modulo $b$ and denote it by $R_b$. Let $O_b$ represent the set of odd numbers in the range $0$ to $b-1$.
    \item From $O_b$, remove shifts of $O_b$ by $2$, then by $2^2$, and so on modulo $b$. If the resulting set is empty, proceed to the next even integer $b + 2$ and repeat the process from Step $(1)$.
\end{enumerate}
\end{algorithm}

The execution time for Algorithm~\ref{ChenQ2algorithm} to process all even integers less than $11184810$ is approximately 5 hours, confirming a positive answer to Chen's Question 2.

For $b = 11184810$, after removing the 24 shifts of $R_{11184810}$ by $2, 2^2, \ldots, 2^{24}$ modulo $11184810$, there are 48 odd residues left. We can check that they match with the arithmetic progressions listed in Theorem~\ref{apthm}. Consequently, all arithmetic progressions $a \pmod{11184810}$ in $\mathcal{U}$ correspond to some \emph{CDL covering system}. The authors wonder whether this holds true in general.

\section{Improvement on $\overline{d}$}~\label{upperboundsec}
\subsection{The algorithms by Habsieger and Roblot}
While the paper by Habsieger and Roblot \cite{HabsiegerRoblot2006} presents an effective algorithm for calculating the upper bound of $\overline{d}$, there are certain inaccuracies and ambiguities in the expression. Therefore, for the sake of clarity and self-explanatory content, we find it necessary to reintroduce the algorithm in our own work.

Let $M$ be a positive integer that is the product of some distinct odd primes. Let
\[
f_M\left( \overline{m} \right) =\left\{ \overline{k}\in \mathbb{Z}/ \text{ord}_2(M) \mathbb{Z}:\overline{m}-2^{\overline{k}}\in \left( \mathbb{Z}/M\mathbb{Z} \right) ^* \right\}
\]
and
\[
\delta _M\left( \nu \right) =\left\{ \overline{m}\in \mathbb{Z}/M\mathbb{Z}:\left| f_M\left( \overline{m} \right) \right|=\nu \right\},
\]
where $\overline{m}$ a residue class modulo $M$ and $0 \leq \nu \leq \text{ord}_2(M)$. Then the following lemma by Habsieger and Roblot provides an upper bound for $\overline{d}$.
\begin{lemma}[{\cite[Lemma 2]{HabsiegerRoblot2006}}]\label{HRlemma}
We have
\[
\overline{d} \leq \sum_{\nu =0}^{\text{ord}_2(M)}{\delta _M\left( \nu \right) \min \left( \frac{1}{M},\frac{2\nu}{\text{ord}_2(M) \varphi \left( M \right) \log 2} \right)}.
\]
\end{lemma}

For example, when $M = 23205 = 3 \times 5 \times 7 \times 13 \times 17$, Lemma~\ref{HRlemma} implies $\overline{d} \leq 0.49250245$. However, when the value of $M$ is large, the algorithm for estimating $\overline{d}$ based on this lemma becomes less efficient.

Certainly, the function $f_M$ could have $2^{\text{ord}_2(M)}$ potential values, all of which are subsets of $\mathbb{Z}/\text{ord}_2(M) \mathbb{Z}$. Nevertheless, it seems that the function $f_M$ takes very few distinct values in practice. So Habsieger and Roblot refined their algorithm.
Let
\[
g_M\left( I \right) =\left\{ \overline{m}\in \mathbb{Z}/M\mathbb{Z}:f_M\left( \overline{m} \right) =I \right\}
\]
and $G_M\left(I\right)=\left|g_M\left(I\right)\right|$ for $I\subset \mathbb{Z}/\text{ord}_2(M) \mathbb{Z}$. Then it is easy to compute values of the function $g_M$ by induction on the number of prime factors of $M$.

Let $M_1,M_2$ be two positive odd integers, with $M_1=pM_2$ for some prime p not dividing $M_2$. For a prime $p$, the image of $f_p$ is clear, we have
\[
I_{p,0}=\mathbb{Z}/ \text{ord}_2(p) \mathbb{Z}
\]
with $G_p(I_{p,0}) = p - \text{ord}_2(p)$ and
\[
I_{p,\overline{j}}=(\mathbb{Z}/ \text{ord}_2(p) \mathbb{Z})\setminus\left\{ \overline{j} \right\}
\]
with $G_p(I_{p,\overline{j}})=1$ for each $\overline{j} \in \mathbb{Z} / \text{ord}_2(p) \mathbb{Z}$.

Let $I_2$ and $I_p$ be in the image of $f_{M_2}$ and $f_p$ respectively. Let $\tilde{I}_2$ and $\tilde{I}_p$ be the subsets of $\mathbb{Z}/ \text{ord}_2(M_1) \mathbb{Z}$ which are inverse images of $I_2$ and $I_p$ under the natural surjections $\mathbb{Z}/ \text{ord}_2(M_1) \mathbb{Z} \to \mathbb{Z}/ \text{ord}_2(M_2) \mathbb{Z}$ and $\mathbb{Z}/ \text{ord}_2(M_1) \mathbb{Z} \to \mathbb{Z}/ \text{ord}_2(p) \mathbb{Z}$ respectively. Then we have a set $I_1$ is in the image of $f_{M_1}$ with
\[
G_{M_1}\left( I_1 \right) =\sum_{I_1=\tilde{I}_2\cap \tilde{I}_p}{G_{M_2}\left( I_2 \right) G_p\left( I_p \right)},
\]
and all sets in the image of $f_{M_1}$ can be obtained this way. Since $\delta_M\left(\nu\right)=\sum_{\left|I\right|=\nu}G_M\left(I\right)$, we can apply Lemma~\ref{HRlemma} to obtain an upper bound estimate of $\overline{d}$.

With $M = 3 \times 5 \times 7 \times 11 \times 13 \times 17 \times 19 \times 31 \times 41 \times 73 \times 241 \times 257$, this refined algorithm implies $\overline{d} < 0.4909409303984105956480078184$. Note that during their computation, sets $I$ with large $G_M(I)$ values were excluded as they were unlikely to contribute significantly to the density, thus easing the computational load. As mentioned in their paper \cite{HabsiegerRoblot2006}, the process took 35 minutes on an Intel Xeon 2.4 GHz processor with a memory stack of 2.1GB. Halbsiger and Roblot also noted that the real limitation lies in memory availability, suggesting that memory constraints may become the primary bottleneck in such calculation.

\subsection{Our enhanced algorithm}
We first present the original version of the algorithm and will later explain how we improved it. Let $p$ be a prime. We define its \emph{admissible cluster} as the following multiset:
\[
\left\{\bigcup_{k=0}^{\text{ord}_2(p)-1}\{k\}\backslash\{0\},
\bigcup_{k=0}^{\text{ord}_2(p)-1}\{k\}\backslash\{1\},
\cdots,
\bigcup_{k=0}^{\text{ord}_2(p)-1}\{k\}\backslash\{\text{ord}_2(p)-1\},
\underbrace{\bigcup_{k=0}^{\text{ord}_2(p)-1}\{k\},\cdots,\bigcup_{k=0}^{\text{ord}_2(p)-1}\{k\}}_{p-\text{ord}_2(p)\text{ times}}\right\}.
\]
For a set $I$, we define its \emph{augmented set} with respect to $(M_2, \ell)$ as
\[
\bigcup_{k=0}^{\displaystyle\frac{\text{ord}_2(M_2 \ell)}{\text{ord}_2(M_2)}}\bigcup_{a\in I}\{a+k\text{ord}_2(M_2)\}.
\]
Now, we can define the \emph{admissible cluster of a set of two primes $\{p_1, p_2\}$}. This cluster is the the intersection of the augmented sets derived from the admissible cluster of $p_1$ with respect to $(p_1, p_2)$ and the augmented sets derived from the admissible cluster of $p_2$ with respect to $(p_2, p_1)$. More generally, we can define the \emph{admissible cluster of a set of primes $\{p_1, p_2, p_3, \cdots \}$}. With this notation, we are ready to present our original algorithm.

\begin{algorithm}\label{upperdensityalgorithm}
We select a set of primes $\mathfrak{P}$ and proceed the following 6 steps to derive an upper bound estimate for $\overline{d}$ corresponding to $\mathfrak{P}$.
\begin{enumerate}
    \item Divide the set of primes $\mathfrak{P}$ manually into two groups, denoted as $\mathfrak{P}_L$ and $\mathfrak{P}_R$.
    \item Generate the admissible clusters for each prime in $\mathfrak{P}_L$.
    \item For the first group of primes $\mathfrak{P}_L  = \{p_1, p_2, p_3, \cdots \}$, begin by identifying the augmented sets of the sets in the admissible cluster of $p_1$ with respect to $(p_1,p_2)$, and similarly, find the augmented sets of the sets in the admissible cluster of $p_2$ with respect to $(p_2,p_1)$. Then, compute their intersection to obtain the admissible cluster of $\{p_1, p_2\}$. Continue this process inductively for the remaining primes in $\mathfrak{P}_L$ to derive the admissible cluster of $\mathfrak{P}_L$.
    \item Perform analogous operations in Step $(2)$ and Step $(3)$ on the primes in $\mathfrak{P}_R$ to compute the admissible cluster of $\mathfrak{P}_R$.
    \item Compute the augmented sets of the sets in the admissible cluster of $\mathfrak{P}_L$ with respect to $\left(\prod_{p\in \mathfrak{P}_L}p,\prod_{p\in \mathfrak{P}_R}p\right)$, and compute the augmented sets of the sets in the admissible cluster of $\mathfrak{P}_R$ with respect to $\left(\prod_{p\in \mathfrak{P}_R}p,\prod_{p\in \mathfrak{P}_L}p\right)$. Then, perform intersection operations between these two sets of augmented sets. Record the number of sets in the resulting intersections.
    \item Utilize the formula $\overline{d} \leq \sum_{\nu=0}^{\text{ord}_2(M)} \delta_M(\nu) \min \left(\frac{1}{M}, \frac{2 \nu}{\text{ord}_2(M) \varphi(M) \log 2}\right)$, where $M$ is the product of the primes in $\mathfrak{P}$, to derive an upper bound of $\overline{d}$.
\end{enumerate}
\end{algorithm}

In Step $(1)$, we divide the elements in set $\mathfrak{P}$ into two subsets, $\mathfrak{P}_L$ and $\mathfrak{P}_R$, to save time and memory, since we do not need to list all sets obtained from the intersection operation in Step $(5)$; instead, we only need the number of elements in the intersection sets. Our rule for partitioning the set $\mathfrak{P}$ is to make the products of elements in the two subsets $\mathfrak{P}_L$ and $\mathfrak{P}_R$ approximately equal.

Next, we explain the enhancements made to Algorithm~\ref{upperdensityalgorithm}.

\subsubsection{Enhancements to improve running speed}
Algorithm~\ref{upperdensityalgorithm} can provide estimates for $\overline{d}$, but it is time-consuming.
Therefore, the following adjustments can be made: in the algorithm, there is no need to use the numbers themselves. Instead, a binary matrix, denoted by $0-1$ entries, can be used to indicate whether a number is in a set, with $1$ representing existence and $0$ representing non-existence. Consequently, the intersection operation of two sets can be considered as multiplying the corresponding rows, which is equivalent to performing the Hadamard product. The step of counting the number of elements can also be regarded as summing up the elements in a row. After implementing these adjustments, it is feasible to integrate the use of Python's NumPy package, enabling multi-core CPU computing.

In testing, compared to the traditional non-matrix calculation method, the enhanced algorithm can accelerate the speed by approximately 1000 times.

\subsubsection{Enhancements to improve memory efficiency}
In practice, the memory usage of the mentioned algorithm is substantial, especially due to the presence of many duplicate sets within the set clusters. For instance, when computing the upper bound for $\overline{d}$ corresponding to the set of 10 primes
\[
\{3, 5, 7, 11, 13, 17, 19, 31, 41, 73\},
\]
the required memory has surpassed 64GB. To conserve memory and expedite the process, we introduce a new single-column matrix, known as the multiplicity matrix, to store the multiplicity of each row's corresponding set in the family of sets. Consequently, when conducting the Hadamard product on two rows, we only need to multiply the corresponding elements in the multiplicity matrix. Furthermore, after each intersection operation between clusters of sets is completed, a deduplication operation is executed. If two rows in the $0-1$ matrix are the same, the rows are merged into one, and the corresponding elements in the multiplicity matrix are summed.

With these improvements, the algorithm can significantly conserve memory. This enables the calculation of the upper bound for $\overline{d}$ corresponding to the set of 12 primes
\[
\{3,5,7,11,13,17,19,31,41,733,241,257\}
\]
in just 24 minutes on a computer with 16GB of memory and an i9-13980HX CPU.

\subsubsection{Utilizing GPU computing to further improve efficiency}
To further optimize the process, we made a slight adjustment to Step $(6)$ of Algorithm~\ref{upperdensityalgorithm}: The large $0-1$ matrix is divided into several blocks and loaded into the GPU, with the size of each block determined by the GPU memory.

By utilizing GPU computation implemented through Python's CuPy package, we can conserve approximately one-third of the memory and significantly increase the speed, with the exact rate depending on the GPU performance. On a computer with 16GB of memory, an i9-13980HX CPU, and an RTX4070 Laptop GPU, the computation with the set of 12 primes
\[
\{3,5,7,11,13,17,19,31,41,73,241,257\}
\]
considered by Habsieger and Roblot can be completed in just 3 minutes, compared to their original time of 35 minutes. Since we do not drop any set $I$ for which $G_M(I)$ is large, our result $\overline{d} \leq 0.49089834$ is slightly better than their $\overline{d} < 0.4909409303984105956480078184$.

Finally, with these enhancements, on a platform with 1536GB of memory, two E5-2697v2 CPUs, and a V100 GPU, it took approximately 167 hours to obtain $\overline{d} < 0.490341088858244$ when using the set
\[
\{3,5,7,11,13,17,19,23,29,31,37,41,61,73\}
\]
for our computation. This proves our Theorem~\ref{upperboundthm}.

\subsection{Our results and some discussion on Theorem~\ref{upperboundthm}} Due to limitations in the NumPy and CuPy packages we utilize, the elements in matrices cannot exceed $2^{63}$. This implies that the product of the selected primes cannot surpass $2^{63}$. Hence, there is a technical aspect to selecting the prime set $\mathfrak{P}$. We explain why we have picked the set
\[
\{3,5,7,11,13,17,19,23,29,31,37,41,61,73\}.
\]
We begin by analysing some data for the set
\[
\{3,5,7,11,13,17,19,31,41,73,241,257\}
\]
used in Habsieger and Roblot's paper \cite{HabsiegerRoblot2006}. We list the process of expanding the set ${3}$ to include the primes mentioned, adding them one by one, along with the corresponding decrease in the upper bound estimates for $\overline{d}$.
\[
\begin{array}{|l|l|l|}
\hline
\text{Set of primes} & \text{Estimates for } \overline{d} & \text{Improvements} \\
\hline
\{3\} & 0.5 &\\
\hline
\{3,5\} & 0.5& 0\\
\hline
\{3,5,7\} & 0.5& 0\\
\hline
\{3,5,7,11\} & 0.49807089 & 0.00192911\\
\hline
\{3,5,7,11,13\} & 0.49621815 & 0.00185274\\
\hline
\{3,5,7,11,13,17\} & 0.49252410 & 0.00369405\\
\hline
\{3,5,7,11,13,17,19\} & 0.49185782 & 0.00066628\\
\hline
\{3,5,7,11,13,17,19,31\} & 0.49143385 & 0.00042397\\
\hline
\{3,5,7,11,13,17,19,31,41\} & 0.49115839 & 0.00027546\\
\hline
\{3,5,7,11,13,17,19,31,41,73\} & 0.49107930 & 0.00007909\\
\hline
\{3,5,7,11,13,17,19,31,41,73,241\} & 0.49098557 & 0.00009373\\
\hline
\{3,5,7,11,13,17,19,31,41,73,241,257\} & 0.49089834 & 0.00008723\\
\hline
\end{array}
\]

We have observed that numbers with $2$ as a primitive root play an important role. Additionally, through our testing, we have observed that the computation time tends to increase as the order of $2$ modulo the product of primes grows larger. Noting that the order of $2$ modulo $3 \times 5 \times 7 \times 11 \times 13 \times 17 \times 19 \times 31 \times 41 \times 73 \times 241 \times 257$ is $720 = 2^4 \times 3^2 \times 5$, we decide to add primes $37$ and $61$ with $\text{ord}_2(37) = 36 = 2^2 \times 3^2$ and $\text{ord}_2(61) = 60 = 2^2 \times 3 \times 5$ to the set,
and remove the prime $257$ with $\text{ord}_2(257) = 16 = 2^4$ from the set, reducing the order of $2$ modulo the product of primes from $720$ to $360$.
\[
\begin{array}{|l|l|l|}
\hline
\{3,5,7,11,13,17,19,31,41,73,241,257\} & 0.49089834 & \text{control group}\\
\hline
\{3,5,7,11,13,17,19,31,37,41,61,73,241\} & 0.49056186 & \text{added $37$ and $61$ and removed $257$}\\
\hline
\end{array}
\]

So far, all primes less than $100$ with a base $2$ order of $360$ or its factors are already in the set. Thus, we can only consider adding primes with base $2$ orders that are not factors of $360$. Since the magnitude of the primes also affects the computation time to a certain extent, we would like to replace $241$ with a smaller prime. Considering the sequence of primes with $2$ as a primitive root $3, 5, 11, 13, 19, 29, 37, 53, 59, 61, \cdots$, we see the relatively small primes in this sequence but not in the set are only $29$ with $\text{ord}_2(29) = 28 = 2^2 \times 7$ and $53$ with $\text{ord}_2(53) = 52 = 2^2 \times 13$. We also consider small primes with a base $2$ order of $11$ or $2^s \times 11$ for some $s \geq 1$, thus we also take $23$ as a candidate. Below is some data for comparison. We find that replacing $241$ with $29$ yields the best result.
\[
\begin{array}{|l|l|l|}
\hline
\{3,5,7,11,13,17,19,31,37,41,61,73,241\} & 0.49056186 & \text{control group}\\
\hline
\{3,5,7,11,13,17,19,29,31,37,41,61,73\} & 0.49041415 & \text{replaced $241$ with $29$}\\
\hline
\{3,5,7,11,13,17,19,31,37,41,53,61,73\} & 0.49060353 & \text{replaced $241$ with $53$}\\
\hline
\{3,5,7,11,13,17,19,23,31,37,41,61,73\} & 0.49062494 & \text{replaced $241$ with $23$}\\
\hline
\end{array}
\]

On the other hand, noting that the current product of primes is already close to $2^{63}$, to ensure the accuracy of the result, we would like to add just one more prime. Since the effects of replacing $241$ with $23$ and $53$ are similar, while $23 < 53$ and $\text{ord}_2(23) < \text{ord}_2(53)$, we decide to add $23$ to save computation time, resulting in the following outcome.
\[
\begin{array}{|l|l|}
\hline
\{3,5,7,11,13,17,19,23,29,31,37,41,61,73\} & \mathbf{0.490341088858244}\\
\hline
\end{array}
\]

We have computed many examples using the algorithm. Now, we analyze the data generated from the computation process and provide some guidance for selecting elements in the set $\mathfrak{P}$. First, if we choose $\mathfrak{P}$ as the set consisting of the first $12$ odd prime numbers, then this gives the best result among all choices of $\mathfrak{P}$ with $12$ odd primes we have tested.
\[
\begin{array}{|l|l|}
\hline
\{3,5,7,11,13,17,19,23,29,31,37,41\} & 0.49064273 \\
\hline
\end{array}
\]
This leads us to believe that the set of first $m$ odd primes will generate the best result among all choices of sets with $m$ odd primes. However, we have the following counterexample.
\begin{counterexample}
For certain values of $m$, some set $\mathfrak{P}$ with $m$ odd primes, which is different from the set of the first $m$ odd primes, may yield a better upper bound estimate for $\overline{d}$.
\[
\begin{array}{|l|l|}
\hline \{3, 5, 7, 11, 13, 17\} & 0.49252410448328 \\
\hline \{3,5,7,13,17,241\} & 0.49243452466582\\
\hline
\end{array}
\]
\end{counterexample}

The reason of this is mainly because $3,5,7,13,17,241$ are all prime factors of $2^{24}-1$, while $\text{ord}_2(3 \times 5 \times 7 \times 11 \times 13 \times 17) = 120$. Therefore, we can guess if $|\mathfrak{P}_1| = |\mathfrak{P}_2|$ and $\text{ord}_2(\prod_{p \in \mathfrak{P}_1} p) < \text{ord}_2(\prod_{p \in \mathfrak{P}_2} p)$, then $\mathfrak{P}_1$ generates a better result than $\mathfrak{P}_2$. Unfortunately, this is still incorrect.
\begin{counterexample}
For some sets of primes $\mathfrak{P}_1$ and $\mathfrak{P}_2$ with $|\mathfrak{P}_1| = |\mathfrak{P}_2|$ and $\text{ord}_2(\prod_{p \in \mathfrak{P}_1} p) < \text{ord}_2(\prod_{p \in \mathfrak{P}_2} p)$, the upper bound for $\overline{d}$ generated from $\mathfrak{P}_1$ may not necessarily be superior to what $\mathfrak{P}_2$ produces.
\[
\begin{array}{|l|l|l|}
\hline \{3, 5, 7, 11, 13, 31, 41, 61, 151, 331, 1321\} & 0.49431157054919&\mathrm{base \ 2 \ order}=60 \\
\hline \{3,5,7,11,13,17,19,31,41,73,241\} & 0.49098556503467&\mathrm{base \ 2 \ order}=360\\
\hline
\end{array}
\]
\end{counterexample}

Finally, we found the following counterexample, which contradicts our intuitive choices of adding primes to a set to produce better estimates.
\begin{counterexample}
For two sets of primes $\mathfrak{P}$ and $\mathfrak{Q}$, if among all $q_i \in \mathfrak{Q}$, the best two results of the upper bounds generated from $\mathfrak{P} \bigcup \{q_i\}$ are $\mathfrak{P} \bigcup \{q_1\}$ and $\mathfrak{P} \bigcup \{q_2\}$, then the upper bound generated from $\mathfrak{P} \bigcup \{q_1,q_2\}$ is not necessarily the best among the upper bounds generated from $\mathfrak{P} \bigcup_{q_1, q_j \in \mathfrak{Q}} \{q_i,q_j\}$. (Although, according to a large number of empirical computation data, the upper bound generated from $\mathfrak{P} \bigcup \{q_1,q_2\}$ often tends to be better.) For example, let $\mathfrak{P}=\{3,5,7,11,17\}$ and $\mathfrak{Q}=\{19,23,29\}$. Then the following table shows the upper bounds for $\overline{d}$ generated from $\mathfrak{P} \bigcup\{19\}$, $\mathfrak{P} \bigcup\{23\}$, and $\mathfrak{P} \bigcup\{29\}$, respectively.
\[
\begin{array}{|l|l|l|}
\hline \{3, 5, 7, 11, 17, 19\}& 0.494609133024577 & \mathrm{best} \\
\hline \{3, 5, 7, 11, 17, 23\}& 0.494870288038247 & \mathrm{second\ best} \\
\hline \{3, 5, 7, 11, 17, 29\}& 0.494883239281366 &  \mathrm{worst} \\
\hline
\end{array}
\]
However, according to the table below, the result generated from $\mathfrak{P} \bigcup\{q_1,q_2\}=\{3,5,7,11,17,19,23\}$ is not the best.
\[
\begin{array}{|l|l|l|}
\hline \{3, 5, 7, 11, 17, 19, 23\}& 0.494486144723180 & \\
\hline \{3, 5, 7, 11, 17, 19, 29\}& 0.494213278918742 &\mathrm{best}\\
\hline \{3, 5, 7, 11, 17, 23, 29\}& 0.494618822711737 & \\
\hline
\end{array}
\]
\end{counterexample}


\bibliographystyle{plain}
\bibliography{bib}

\end{document}